\documentclass[10pt,a4paper,pdfspacing]{scrartcl} 
\usepackage{url}
\usepackage{amsopn,amsmath,amssymb,amsthm}
\usepackage[nochapters]{classicthesis-ldpkg}
\usepackage[nochapters,drafting]{classicthesis} 

\theoremstyle{remark}
\newtheorem{remark}{Remark}[section]
\theoremstyle{plain}
\newtheorem{theorem}[remark]{Theorem}
\newtheorem{proposition}[remark]{Proposition}
\newtheorem{lemma}[remark]{Lemma}
\newtheorem{corollary}[remark]{Corollary}

\newenvironment{grazie}{\textsc{Acknowledgements.} \upshape}{\hfill\par}

\areaset{12cm}{22cm}

\begin{document}
    \title{\rmfamily\normalfont\spacedallcaps{Increasing variational solutions for a nonlinear $p$-laplace equation without growth conditions}}
    \author{\spacedlowsmallcaps{simone secchi}\thanks{Dipartimento di Matematica ed Applicazioni, Universit\`a di Milano--Bicocca, via R.~Cozzi 53, edificio U5, I-20125 Milano (Italy). E-mail: \texttt{simone.secchi@unimib.it}}}

\date{} 
        
    \maketitle
    
    \begin{abstract}
        By means of a recent variational technique, we prove the existence of radially monotone solutions to a class of nonlinear problems involving the $p$-Laplace operator. No subcriticality condition (in the sense of Sobolev spaces) is required.
    \end{abstract}

\section{Introduction}

Variational methods are a powerful tool for solving partial differential equations. Not only do they provide existence theorems, but they can often provide solutions with additional properties. For instance, by Palais' Principle of Symmetric Criticality (see \cite{pal}) it is well known that radial symmetry is a \emph{natural constraint} in Critical Point Theory. Roughly speaking, if a variational equation is invariant under rotation, then we can look for radially symmetric weak solutions (\ie\ weak solutions found as critical points of a suitable functional) by simply restricting the variational framework to the smaller subspace of radial functions. 

Once we have a radial solution, a natural question is whether this must be a monotone function. Although many tools for proving the monotonicity of radial solutions are available (the Gidas--Ni--Nirenberg  theory \cite{gnn}, or the use of some kind of symmetrization), this is not a natural constraint.

  In a recent paper \cite{St},  Serra \emph{et al.} introduced a new and interesting variational scheme to find increasing positive and radial  solutions to a semilinear elliptic equation with Neumann boundary conditions on a ball in $\mathbb{R}^N$. More precisely, they studied the problem
  \[
        \begin{cases}
        -\Delta u + u =  a(|x|)f(u) &\text{in $B_R$} \\
        u>0 &\text{in $B$}\\
        \frac{\partial u}{\partial \nu} = 0 &\text{on $\partial B_R$},
\end{cases}
        \]
where $a \colon [0,+\infty)  \to \mathbb{R}$ and $f \colon \mathbb{R} \to \mathbb{R}$ are regular functions which satisfy very mild assumptions. In particular, no growth condition like \(f(u) \leq  C u^q\) with a subcritical exponent \(q<(N+2)/(N-2)\) is required. Semilinear problems like this appear in some applications: we refer to \cite{bss,l1,lnt,ay1,ay2} and to the references therein. In some situations, the structure of the equation or numerical experiments suggest the existence of \emph{increasing} radial solutions: this happens for the H\'{e}non equation with Neumann boundary conditions, where $a(x)=|x|^\alpha$ for some $\alpha>0$. It is difficult to embed this monotonicity in the variational setting, and yet these solutions do exist, as proved in \cite{bss} by means of the shooting method. The variational approach developed in \cite{St} allows us to deal with these solutions in a clean and elementary way.

\bigskip

The purpose of our note is to find radially increasing solutions to a class of quasilinear equations involving the $p$-Laplace operator.

In the first part we will deal with the nonlinear eigenvalue problem for the \(p\)--Laplace operator
\begin{equation} \label{1}
        \begin{cases}
        -\Delta_p u + |u|^{p-2}u = \lambda a(|x|)f(u) &\text{in $B$} \\
        u>0 &\text{in $B$}\\
        \frac{\partial u}{\partial \nu} = 0 &\text{on $\partial B$},
\end{cases}
\end{equation}
where $B=\{x \in \mathbb{R}^N  \colon |x| <1 \}$ and $N \geq 3$.
Here we want $\lambda \in \mathbb{R}$ and a positive $u \in W^{1,p}(B)$ such that (\ref{1}) is solved in the usual weak sense. We will borrow some ideas from \cite{St} to solve a constrained optimization problem.

In the second part, we will look for solutions of (\ref{1}) when $\lambda$ is \emph{fixed}, \ie\ it is not an unknown.

\begin{remark}
Since $a$ is non--constant and we do not suppose that $f$ is a homogeneous function, the nonlinear eigenvalue problem is not equivalent to the problem with a fixed $\lambda$.
\end{remark}

\noindent\textbf{Notation}

\begin{itemize}
\item The $p$--Laplace operator is formally defined by $\Delta_p u = \operatorname{div}(|\nabla u|^{p-2}\nabla u)$. We will always assume that $1<p<\infty$.
\item $W^{1,p}(B)$ is the usual Sobolev space endowed with the norm
  $\|u\| = \left( \int_B |\nabla u|^p + |u|^p \right)^{1/p}$. If $u\in
  W^{1,p}(B)$, its positive part $u^{+} = \max\{u,0\}$ belongs to
  $W^{1,p}(B)$.
\item $\nu$ is the outer normal vector.
\item By a useful abuse of notation, we will often identify a radial
  function in $\mathbb{R}^N$ with its radial representative. If $u$ is
  a radial function in $\mathbb{R}^N$, we will often write
  $u(x)=u(|x|)=u(r)$, where $r=|x|$. 
\item In some formul\ae, we will often write $|u|^p$ instead of $u^p$ even when $u\geq 0$. This choice preserves some symmetry in the equations.
\item The notation $|S^{N-1}|$ is used for the Lebesgue measure of the sphere $S^{N-1}$.
\item The abbreviation ``a.e.'' stands for ``almost everywhere''.
\end{itemize}

\section{The nonlinear eigenvalue problem}

In the rest of this section, we will retain some assumptions on $a$
and $f$ in (\ref{1}):
\begin{description}
\item[(A)] $a \in L^1(B)$ is a non-costant radially increasing function,
  satisfying $a(r)>0$ for almost every $r \in [0,1]$;
\item[(F)] $f\in C^1([0,+\infty))$ is a positive function that
  satisfies $f(0)=0$; moreover $t \mapsto f(t)/t^{p-1}$ is strictly
  increasing on $(0,+\infty)$.
\end{description}

\begin{remark}
  Condition~(F) implies that $f$ is strictly increasing. Being
  positive away from $t=0$, the primitive $F(t)=\int_0^r f(s)\, ds$ is
  strictly increasing as well.
\end{remark}

\begin{remark}
By inspecting the rest of this paper, one will notice that the integrability assumption on $a$ could be slightly weakened. We required in (A) that $\int_B a(|x|)\, dx = |S^{N-1}| \int_0^1 a(r) r^{N-1}\, dr < \infty$, but it suffices to assume that $\int_0^1 a(r)\, dr < \infty$ as in \cite{St}.
\end{remark}

We state our main result.

\begin{theorem} \label{main} 
The nonlinear eigenvalue problem
  (\ref{1}) has at least one radially symmetric solution, and this
  solution is a monotone increasing function.
\end{theorem}

In the particular case of the pure--power nonlinearity $f(u)=u^q$, our existence theorem reads as follows.
\begin{corollary}
For any $q>1$, the nonlinear eigenvalue problem
\begin{equation*} 
        \begin{cases}
        -\Delta_p u + |u|^{p-2}u = \lambda a(|x|)|u|^q &\text{in $B$} \\
        u>0 &\text{in $B$}\\
        \frac{\partial u}{\partial \nu} = 0 &\text{on $\partial B$}
\end{cases}
\end{equation*}
has at least one radially symmetric solution, and this solution is a monotone increasing function.
\end{corollary}

\subsection{Introducing a variational problem}

Since we are looking for radially symmetric increasing solutions, we define the set
\begin{equation*}
  \mathcal{M} = \left\{ u \in W^{1,p}(B) \mid \text{$u$ is a radially increasing function, $u \geq 0$ a.e.} \right\}.
\end{equation*}
More explicitly, the elements of $\mathcal{M}$ are those functions $u$ from $W^{1,p}(B)$ that are invariant under any rotation in $\mathbb{R}^N$ and $u(|x|) \leq u(|y|)$ whenever $|x| < |y|$.

A reasonable attempt to solve the nonlinear eigenvalue problem (\ref{1}) is to find a solution for the variational problem
\begin{equation} \label{2} S = \sup \left\{ \int_B a(|x|) F(u) \, dx
    \mid \text{$u \in \mathcal{M}$, $\|u\|^p=1$} \right\},
\end{equation}
where $F(r) = \int_0^r f(s) \, ds$. If $\mathcal{M}$ were a smooth submanifold of codimension one, then we could refer to the classical theory of Lagrange multipliers, and conclude.  Since this is not the case, in section \ref{sec:3} we will prove directly that any solution of (\ref{2}) gives rise to a solution of (\ref{1}).

Turning back to the definition of the set $\mathcal{M}$, we notice that any $u \in \mathcal{M}$ is (identified with) a continuous function on $[0,1]$; indeed, by monotonicity, we can set $u(0)=\lim_{r \to 0+}u(r)$.

\begin{remark}
  The set $\mathcal{M}$ is indeed a (closed) cone in $W^{1,p}(B)$. It
  is tempting to solve (\ref{1}) by applying some Critical Point
  Theory on metric spaces to the free functional
\[
u \mapsto \frac{1}{p} \int_B (|\nabla u|^p + |u|^p ) - \int_B a(|x|)
F(u) \, dx.
\]
However, we are unable to pursue this idea further.
\end{remark}

The main advantage of working with $\mathcal{M}$ is that it consists
of bounded functions.

\begin{proposition}\label{prop:infinity}
  There exists a positive constant $C_{\mathcal{M}}$ such that $\|u\|_{L^\infty} \leq C
  \|u\|$ for any $u \in \mathcal{M}$. In particular, strong convergence in $\mathcal{M}$ implies
  uniform convergence.
  \end{proposition}
\begin{proof}
  For any $0<r<1$, the fact that $u$ is non-negative and increasing
  implies that $\max_{x \in B} |u(x)| = \max_{x \in B \setminus
    \overline{B(0,r)}} |u(x)|$. By a straightforward modification of the radial Lemma 2.1 proved in
  \cite{gs}, there exists a positive constant, independent of $u$,
  such that
\[
\max_{x \in B \setminus \overline{B(0,r)}} |u(x)| \leq C
\|u\|_{W^{1,p}(B \setminus \overline{B(0,r)})} \leq C \|u\|.
\] 
The proof is complete.
\end{proof}

We define the functional
\begin{equation*}
  I(u)=\int_B a(|x|) F(u)\, dx,
\end{equation*}
and we notice that, since $F$ is increasing, 
\[
I(u) =\int_B a(|x|) F(u)\, dx \leq \int_B a(|x|) F(\|u\|_{L^\infty}) \, dx = F(\|u\|_{L^\infty}) \|a\|_{L^1(B)},
\]
for any $u \in \mathcal{M}$. Therefore $I$ is well defined on $\mathcal{M}$.

\begin{remark}
  We do not really need the monotonicity of $F$, since $F$ is a
  continuous function and $u$ is in $L^\infty(B)$. However, we
  highlight that $I$ cannot be defined on the whole space~$W^{1,p}(B)$,
  since we have no growth limitation on $F$.
\end{remark}

\begin{lemma}
  The quantity $S$ is a finite positive number.
\end{lemma}
\begin{proof}
  Indeed, for any $u \in \mathcal{M}$ such that $\|u\|^p=1$, we can
  use Proposition \ref{prop:infinity}
\begin{eqnarray*}
  \int_B a(|x|) F(u)\, dx &\leq& \int_B a(|x|) F(\|u\|_{L^\infty}) \, dx = F(\|u\|_{L^\infty}) \|a\|_{L^1(B)} \\
  &\leq& F(C_{\mathcal{M}}) \|a\|_{L^1(B)}.
\end{eqnarray*}

The fact that $S>0$ is a trivial consequence of the definition of $S$ and of the strict positivity of $a$ and $F$ away from zero.

\end{proof}

\begin{proposition}
The value $S$ is attained.
\end{proposition}
\begin{proof}
  We take any sequence $\{u_n\}_{n=1}^\infty$ in $\mathcal{M}$ such
  that $\|u_n\|^p=1$ for each $n \geq 1$ and $I(u_n) \to S$. In
  particular this sequence is bounded in $W^{1,p}(B)$, a reflexive
  Banach space. We can assume, up to a subsequence, that $u_n$
  converges weakly in $W^{1,p}(B)$ and pointwise almost everywhere to
  some $u \in W^{1,p}(B)$. It is easily checked that $u \in
  \mathcal{M}$. Since $\sup_n \|u_n\|_{L^\infty} < +\infty$ by Proposition
  \ref{prop:infinity}, by Lebesgue's theorem on dominated convergence,
  $S=\lim_n I(u_n) = \int_B a(|x|) F(u)\, dx$.  Now, if $\|u\|^p=1$,
  we are done: $S$ is attained at $u$. By the weak lower
  semicontinuity of the norm, $\|u\|^p \leq \liminf_n
  \|u_n\|^p=1$. The case $u=0$ is excluded, since $\{u_n\}$ would then
  converge strongly to zero, and therefore $S=0$. Assume that
  $0<\|u\|^p<1$. Then $\tilde{u}=u/\|u\|^p$ lies in $\mathcal{M}$ and
  $\|\tilde{u}\|^p=1$. Therefore, the strict monotonicity of $F$
  implies
\[
I(\tilde{u}) = \int_B a(|x|) F\left(\|u\|^{-p} u(|x|)\right)\, dx >
\int_B a(|x|) F(u(|x|))\, dx = S.
\]
This contradiction shows that $\|u\|^p=1$, and in particular $\lim_n
\|u_n\|^p = \|u\|^p$. This means also that $\{u_n\}$ converges to $u$
strongly in $W^{1,p}(B)$.
\end{proof}

\subsection{Maximizers are solutions of the differential equation} \label{sec:3}

In the last section we proved that $S$ is attained by a radially
symmetric function $u$, non-negative and monotone increasing.  
In the standard approach of Critical Point Theory, this $u$ would be a critical point of $I$
constrained to $\mathcal{M}$, and there would exist a Lagrange multiplier.
Our situation is different, since we cannot compare $I(u)$ to $I(u+\varepsilon v)$ for any (radially symmetric)  $v\in W^{1,p}(B)$. Indeed, $u+\varepsilon v$ need not be positive, nor increasing.

However $u$ is, in a very weak sense, a solution to a differential
inequality.

\begin{proposition}\label{prop:10}
  If $v \in W^{1,p}(B)$ is a radial function such that $u+\varepsilon
  v \in \mathcal{M}$ for any~$\varepsilon \ll 1$, then there exists
  $\lambda=\lambda(u)>0$ such that
\begin{equation}\label{supsol}
\int_B |\nabla u|^{p-2} \nabla u \cdot \nabla v +  |u|^{p-2} u v \, dx \geq \lambda \int_B a(|x|)f(u)v \, dx.
\end{equation}
\end{proposition}
\begin{proof}
  By assumption, $(u+\varepsilon v)/ \|u+\varepsilon v\|^p$ is an
  admissible function for the optimization problem
  (\ref{2}). Therefore, the auxiliary function
\[
G(\varepsilon) = \int_B a(|x|) F \left( \frac{u+\varepsilon v}{\|u+\varepsilon v\|^p} \right) \, dx
\]
attains a maximum at $\varepsilon=0$. By a direct calculation, and recalling that $\|u\|^p=1$, we find
\begin{multline*}
0 \geq G'(0)  =\\
=\int_B a(|x|) f(u) \left( v - 
\left(
\int_B |\nabla u|^{p-2} \nabla u \cdot \nabla v + \int_B |u|^{p-2}uv
\right)
u
\right) \, dx.
\end{multline*}
We can rewrite this inequality in the form
\[
\int_B |\nabla u|^{p-2} \nabla u \cdot \nabla v + \int_B |u|^{p-2}uv
\geq \frac{1}{\int_B a(|x|)f(u)u} \int_B a(|x|)f(u)v,
\]
and the conclusion follows by setting $\lambda=1/\int_B a(|x|)f(u)u$.
\end{proof}
Of course, if (\ref{supsol}) holds true for any radial element $v \in
W^{1,p}(B)$, then we can conclude that $u$ is a weak supersolution to
(\ref{2}). This is indeed true, but we need some more work. We begin
with a sort of ``maximum principle''. This not obvious, since $u$ is
not (yet) a \emph{solution} of an equation.

\begin{lemma} \label{lem:pos}
The function $u$ is strictly positive in $\overline{B}$.
\end{lemma}
\begin{proof}
  Consider the auxiliary function $\varphi (x)=e^{|x|}$, defined for
  all $x \in \mathbb{R}^N$. It is easy to show that $\varphi \in
  W^{1,p}(B)$. It is of course a positive radial function, and
  monotone increasing. We are going to prove that $u(x) \geq \kappa
  \varphi(x)$ for all~$x \in \overline{B}$, provided $\kappa>0$ is
  chosen suitably. Fix an arbitrary radial function $\psi \in
  W^{1,p}(B)$ with $\psi \geq 0$ and $\psi_{|\partial B}=0$. Then,
  denoting by $|S^{N-1}|$ the Lebesgue measure of the unit sphere in
  $\mathbb{R}^N$,
\begin{eqnarray*}
  \int_B |\nabla \varphi|^{p-2} \nabla \varphi \cdot \nabla \psi \, dx &=& \int_B e^{(p-1)|x|} \frac{x}{|x|} \cdot \nabla \psi \, dx \\
  &=& |S^{N-1}| \int_0^1 e^{(p-1)r} \psi'(r) r^{N-1}\, dr \\
  &=& -  |S^{N-1}| \int_0^1 \left(e^{(p-1)r} r^{N-1} \right)' \psi(r) \, dr,
\end{eqnarray*}
so that
\begin{multline} \label{4}
  \int_B |\nabla \varphi|^{p-2} \nabla \varphi \cdot \nabla \psi \, dx + \int_B |\varphi|^{p-2}\varphi \psi \,dx  \\
  = -  |S^{N-1}| \int_0^1 \left(e^{(p-1)r} r^{N-1} \right)' \psi(r) \, dr + |S^{N-1}| \int_0^1 e^{(p-1)r} \psi(r) \, dr \\
  = |S^{N-1}| \int_0^1 e^{(p-1)r} \left( 1-N+(2-p)r \right)\, dr < 0
\end{multline}
since $N \geq 3$. Now we choose $\kappa>0$ such that $\kappa \varphi$
equals $u$ on the boundary. Explicitely, $\kappa = u(1)/e$, where we
have denoted by $u(1)$ the constant value of $u$ on $\partial B$. From
now on, we will write again $\varphi$ instead of $\kappa \varphi$.
We
now apply (\ref{4}) to $\psi = (\varphi-u)^{+} \in W_0^{1,p}(B)$. This
function is radially symmetric and positive in $B$. From
(\ref{supsol}) we get
\begin{equation} \label{5}
\int_B |\nabla u|^{p-2} \nabla u \cdot \nabla (\varphi - u)^{+} + \int_B |u|^{p-2} u (\varphi - u)^{+} \geq \lambda \int_B a(|x|)f(u)(\varphi - u)^{+},
\end{equation}
whereas from (\ref{4}) we get
\begin{equation} \label{6} \int_B |\nabla \varphi|^{p-2} \nabla
  \varphi \cdot \nabla (\varphi - u)^{+} \, dx + \int_B
  |\varphi|^{p-2}\varphi (\varphi - u)^{+} \,dx <0.
\end{equation}
Subtracting (\ref{6}) from (\ref{5}) we get
\begin{multline}
  \int_B \left( |\nabla u|^{p-2}\nabla u - |\nabla \varphi|^{p-2} \nabla \varphi \right) \cdot \nabla (\varphi-u)^{+} + \\
  {}+ \int_B \left(|u|^{p-2}u - |\varphi|^{p-2} \varphi \right)
  (\varphi - u)^{+} \geq 0.
\end{multline}
Applying Lemma \ref{lem:elem} below, we deduce that
\begin{multline*}
  0 \leq \int_B \left( |\nabla u|^{p-2}\nabla u - |\nabla \varphi|^{p-2} \nabla \varphi \right) \cdot \nabla (\varphi-u)^{+} + \\
  {}+ \int_B \left(|u|^{p-2}u - |\varphi|^{p-2} \varphi \right) (\varphi - u)^{+} \\
  \leq - c_p \|(\varphi-u)^{+}\|^p \quad\text{if $p \geq 2$},
\end{multline*}
and 
\begin{eqnarray*}
  0 &\leq& \int_B \left( |\nabla u|^{p-2}\nabla u - |\nabla \varphi|^{p-2} \nabla \varphi \right) \cdot \nabla (\varphi-u)^{+} + \\
  &&{}+ \int_B \left(|u|^{p-2}u - |\varphi|^{p-2} \varphi \right) (\varphi - u)^{+} \\
  &\leq& - c_p \left( \int_B \frac{|\nabla (\varphi-u)^{+}|^2}{(|\nabla \varphi|+|\nabla u|)^{2-p}} + \int_B \frac{|(\varphi-u)^{+}|^2}{(|\varphi|+|u|)^{2-p}} \right) \quad\text{if $1<p<2$}.
\end{eqnarray*}
for some constant~$c_p>0$.  This implies in both cases that
$(\varphi-u)^{+}=0$, and we conclude that $u \geq \varphi$ in
$\overline{B}$.
\end{proof}

\begin{lemma}[\cite{simon}, \cite{l}] \label{lem:elem}
Given $1<p<\infty$, there exists a universal constant $c_p>0$ such that
\[
\langle |x|^{p-2}x-|y|^{p-2}y, x-y \rangle \geq 
\left\lbrace
\begin{array}{ll}
c_p |x-y|^p &\text{if $p\geq 2$} \\
c_p \frac{|x-y|^2}{(|x|+|y|)^{2-p}} &\text{if $1<p<2$}
\end{array}
\right.
\]
for any $x$, $y \in \mathbb{R}^N$.
\end{lemma}

\begin{remark}
A slightly different proof of Lemma \ref{lem:pos} can be easily obtained by modifying the ideas of \cite[Lemma 6]{St}. Indeed, consider the problem
\begin{equation*}
\left\lbrace
\begin{array}{ll}
-\Delta_p \varphi +|\varphi|^{p-2}\varphi =0 &\text{in $B$} \\
\varphi=u &\text{on $\partial B.$}
\end{array}
\right.
\end{equation*}
By standard arguments (see \cite{hk}) there is one and only one
solution $\varphi \in W^{1,p}(B)$ such that $\varphi-u \in
W_0^{1,p}(B)$. By the version of Palais' Principle of Symmetric
Criticality for Banach spaces (see \cite{pal}) this solution must
coincide with the unique solution of $\inf \{ \|\varphi\|^p \mid
\text{$\varphi-u \in W_0^{1,p}(B)$, $\varphi$ radial}\}$, and
therefore $\varphi$ is a radial function. Since $u>0$ on $\partial B$,
we have $\varphi \geq 0$ in $B$, and by the strong maximum principle
(see \cite{v}) we actually have $u>0$ in $B$. Then $\varphi$ solves
the ordinary differential equation
$(r^{N-1}|\varphi'|^{p-2}\varphi')'=|\varphi|^{p-2}\varphi$ for $r \in
[0,1]$. Setting $w(r)=|\varphi'|^{p-2}\varphi'$ we can check that
$w(0)=0$ and
\[
w'+\frac{N-1}{r}w=|\varphi|^{p-2}\varphi.
\]
Hence
\[
w(r)=r^{1-N} \int_0^r |\varphi(s)|^{p-2}\varphi(s)s^{N-1}\, ds,
\]
and the strict positivity of $\varphi$ implies that $w>0$, \ie\ $\varphi$ is a strictly increasing function. In particular $\varphi-u \in \mathcal{M}$, and we can again find that (\ref{6}) holds true. The we can proceed as in the proof of Lemma \ref{lem:pos} and conclude that $u \geq \varphi$.
\end{remark}

We observe that our maximizer $u$ belongs to $\mathcal{M}$, and is therefore a monotone increasing radial function. Actually, $u$ is a.e. strictly increasing, and this will be a useful piece of information in the next pages.

\begin{lemma}
Any maximizer $u$ for $S$ satisfies
\begin{equation}\label{8}
u'(r)>0 \quad\text{for almost every $r \in (0,1)$}.
\end{equation}
\end{lemma}
\begin{proof}
Pick a number $\rho \in (0,1)$. For the time being, let $\mu >0$ a parameter that we will choose suitably, and let $\delta>0$. The function $v_\delta \colon [0,1] \to \mathbb{R}$, defined by the formula
\begin{equation*}
v_\delta(r)=
\begin{cases}
-1 &\text{if $0 \leq r \leq \rho$}\\
-1-\frac{1+\delta}{\rho}(r-\rho) &\text{if $\rho<r \leq \rho+\delta$}\\
\mu &\text{if $\rho+\delta<r \leq 1$},
\end{cases}
\end{equation*}
is continuous. Clearly $x \mapsto v_\delta(|x|)$ belongs to $W^{1,p}(B)$. For every positive and \emph{small} number $s$, we can check that $u+s v_\delta \in \mathcal{M}$. By (\ref{supsol}), 
\begin{equation} \label{9}
\int_B |\nabla u|^{p-2} \nabla u \cdot \nabla v_\delta +  |u|^{p-2} u v_\delta \, dx \geq \lambda \int_B a(|x|)f(u)v_\delta \, dx.
\end{equation}
Pointwise and in every Lebesgue space $L^q(B)$ with $q<\infty$, as $\delta \to 0^{+}$ the function $v_\delta$ tends to
\[
v(r)=
\begin{cases}
-1 &\text{if $0 \leq r \leq \rho$}\\
\mu &\text{if $\rho < r \leq 1$}.
\end{cases}
\]
Therefore
\[
\lim_{\delta \to 0+} \int_B |u|^{p-2} u v_\delta = - \int_{B(0,\rho)} |u|^{p-2}u + \mu \int_{B\setminus B(0,\rho)} |u|^{p-2}u
\]
and
\[
\lim_{\delta \to 0+} \int_B a(|x|) f(u)v_\delta \, dx = -\int_{B(0,\rho)} a(|x|) f(u)\, dx + \mu \int_{B\setminus B(0,\rho)} a(|x|) f(u)\, dx.
\]
On the other hand, since $|u'|^{p-2}u' \in L^1_{\textrm{loc}}(0,1)$, 
\begin{eqnarray*}
\int_B |\nabla u|^{p-2}\nabla u \cdot \nabla v_\delta &=& |S^{N-1}| \int_0^1 |u'(r)|^{p-2}u'(r) v_\delta'(r) r^{N-1}\, dr \\
&=& \frac{1+\mu}{\delta} |S^{N-1}| \int_\rho^{\rho+\delta}  |u'(r)|^{p-2}u'(r) r^{N-1}\, dr \\
&\leq& |S^{N-1}| (1+\mu) (\rho+\delta)^{N-1} \frac{1}{\delta} \int_\rho^{\rho+\delta} |u'(r)|^{p-2}u'(r)\, dr \\
&=& |S^{N-1}| (1+\mu)\rho^{N-1} |u'(\rho)|^{p-2}u'(\rho) + o(1)
\end{eqnarray*}
for almost every $\rho \in (0,1)$, as $\delta \to 0^{+}$. If we let $\delta \to 0^{+}$ in (\ref{9}) and use the last three relations, we conclude that
\begin{multline} \label{10}
(1+\mu)\rho^{N-1} |u'(\rho)|^{p-2}u'(\rho) \\
\geq \int_0^\rho |u(r)|^{p-2} u(r) r^{N-1} \, dr - \mu \int_\rho^1 |u(r)|^{p-2} u(r) r^{N-1} \, dr \\
{}- \lambda \int_0^\rho a(r)f(u)r^{N-1}\, dr + \lambda \mu \int_\rho^1 a(r)f(u)r^{N-1}\, dr.
\end{multline}
Since the function $s \mapsto f(s)/s^{p-1}$ is strictly increasing, we can write
\begin{equation*}
a(r)f(u(r))=a(r)\frac{f(u(r))}{u(r)^{p-1}} u(r)^{p-1} \leq a(\rho) \frac{f(u(\rho))}{u(\rho)^{p-1}} u(r)^{p-1}
\end{equation*}
for a.e. $r \in [0,\rho]$, and
\begin{equation*}
a(r)f(u(r))=a(r)\frac{f(u(r))}{u(r)^{p-1}} u(r)^{p-1} \geq a(\rho) \frac{f(u(\rho))}{u(\rho)^{p-1}} u(r)^{p-1}
\end{equation*}
for a.e. $r \in [\rho,1]$. Hence we obtain
\begin{equation*}
-\int_0^\rho a(r) f(u) r^{N-1}\, dr \geq -a(\rho) \frac{f(u(\rho))}{u(\rho)^{p-1}} \int_0^\rho u(r)^{p-1} r^{N-1}\, dr
\end{equation*}
and
\begin{equation*}
\mu \int_\rho^1 a(r) f(u) r^{N-1}\, dr \geq \mu a(\rho) \frac{f(u(\rho))}{u(\rho)^{p-1}} \int_\rho^1 u(r)^{p-1} r^{N-1}\, dr.
\end{equation*}
Now, at least one of these two inequalities must be strict.
Otherwise we would deduce that
\begin{equation*}
a(r) \frac{f(u(r))}{u(r)^{p-1}} = a(\rho) \frac{f(u(\rho))}{u(\rho)^{p-1}} 
\end{equation*}
for a.e. $r \in [0,1]$. Then $u$ would be a constant function, and so would be $a$. But this is in contradiction with our assumptions.

Going back to (\ref{10}), we have proved that
\begin{multline} \label{11}
(1+\mu)\rho^{N-1} |u'(\rho)|^{p-2}u'(\rho) \\
> \int_0^\rho |u(r)|^{p-2} u(r) r^{N-1} \, dr - \mu \int_\rho^1 |u(r)|^{p-2} u(r) r^{N-1} \, dr \\
{}+ a(\rho) \frac{f(u(\rho))}{u(\rho)^{p-1}} \left( -\int_0^\rho u(r)^{p-1}r^{N-1}\, dr + \mu \int_\rho^1 u(r)^{p-1}r^{N-1}\, dr \right).
\end{multline}
If $\mu$ solves the equation
\begin{equation*}
\int_0^\rho u(r)^{p-1}r^{N-1}\, dr = \mu \int_\rho^1 u(r)^{p-1}r^{N-1}\, dr,
\end{equation*}
then (\ref{11}) becomes $u'(r)>0$ for a.e. $r \in [0,1]$.
\end{proof}
\begin{remark}
We have seen that the strict monotonicity of $u$ a.e. depends on the assumption that $a$ is non--constant. On the other hand, suppose that $a \equiv 1$. If the equation $f(s)=s^{p-1}$ possesses a solution $s_0>0$, then our problem (\ref{1}) is solved by the constant function $u(x)=s_0$ (with $\lambda=1$). In particular, when $f(s)=s^q$, there always exists a positive solution $u$ that does not satisfy $u'(r)>0$ a.e. We do not know if there exist non--constant increasing solutions that do not satisfy the strict monotonicity a.e.
\end{remark}
We can now prove that the maximizer $u$ is actually a weak solution of our original problem.
\begin{proposition}
Any maximizer $u$ for $S$ satisfies
\begin{equation*}
\int_B |\nabla u|^{p-2}\nabla u \cdot \nabla v + \int_B |u|^{p-2}u v = \lambda \int_B a(|x|) f(u)v\, dx
\end{equation*}
for any radial function $v \in W^{1,p}(B)$.
\end{proposition}
\begin{proof}
We notice that it is enough to prove that
\begin{equation*}
\int_B |\nabla u|^{p-2}\nabla u \cdot \nabla v + \int_B |u|^{p-2}u v \geq \lambda \int_B a(|x|) f(u)v\, dx
\end{equation*}
for any radial function $v \in W^{1,p}(B)$, because this inequality is odd in $v$. Moreover, by density, we can assume that $v$ is (radial and) of class $C^1$.

Then we can introduce the sets
\[
\Omega_k = \left\{ r \in (0,1) \mid \text{$u'(r)$ exists and $u'(r)>1/k$} \right\}, \quad k=1,2,\ldots
\]
It is easy to check that the Lebesgue measure of $[0,1] \setminus \bigcup_k \Omega_k$ is zero by the previous Lemma. Denote, as usual, by $\chi_{\Omega_k}$ the characteristic function of $\Omega_k$. For any radial $v \in C^1(0,1)$, we define
\[
v_k(r)=v(0)+\int_0^r v'(s) \chi_{\Omega_k}(s)\, ds,
\]
so that $v_k \in W^{1,\infty}(0,1)$ and $v'_k(r)=v'(r) \chi_{\Omega_k}(r)$ for a.e. $r \in [0,1]$.

Fix $k\geq 1$ and pick a small $\varepsilon>0$. We claim that
\begin{equation} \label{12}
u+ \varepsilon v_k \in \mathcal{M} \setminus \{0\}.
\end{equation}
Indeed, for every $r \in [0,1]$,
\[
u(r)+\varepsilon v_k(r) \geq u(0)-\varepsilon \|v\|_{L^\infty} > 0
\]
provided $\varepsilon$ is small enough. Moreover, for almost every $r \in [0,1]$ we have by definition
\[
\frac{d}{dr} (u+\varepsilon v_k) = u'(r)+\varepsilon v'(r) \chi_{\Omega_k}(r).
\]
When $r \in \Omega_k$,
\begin{equation*}
\frac{d}{dr} (u+\varepsilon v_k) = u'(r)+\varepsilon v'(r) \chi_{\Omega_k}(r) = u'(r)+\varepsilon v'(r) 
> \frac{1}{k} - \varepsilon \|v'\|_{L^\infty}> 0
\end{equation*}
provided $\varepsilon$ is small enough. When $r \notin \Omega_k$, 
\[
\frac{d}{dr} (u+\varepsilon v_k) = u'(r)
\]
and we know that $u'>0$ almost everywhere. We conclude that the derivative of $u+\varepsilon v_k$ is almost everywhere strictly positive, and this implies the strict monotonicity of $u+\varepsilon v_k$. Hence by (\ref{supsol})
\begin{multline*}
\int_0^1 u'(r)^{p-2} u'(r) v_k'(r) r^{N-1}\, dr + \int_0^1 u(r)^{p-1} v_k(r) r^{N-1}\, dr \\
\geq \lambda \int_0^1 a(r)f(u(r)) v_k(r) r^{N-1}\, dr
\end{multline*}
for any $k\geq 1$. Since $|v_k(r)| \leq |v(0)|+\|v'\|_{L^\infty}$ and $\|v'_k\|_{L^\infty} \leq \|v'\|_{L^\infty}$, by the Ascoli--Arzel\`a theorem $v_k$ converges uniformly to some $v$ on $[0,1]$. Since $\Omega_k \subset \Omega_{k+1}$, $v'_k(r) \to v'(r)$ as $k \to +\infty$ for almost every $r \in (0,1)$. An application of Lebesgue's theorem on Dominated Convergence implies now
\begin{multline*}
\int_0^1 |u'(r)|^{p-2} u'(r) v'(r) r^{N-1}\, dr + \int_0^1 u(r)^{p-1} v(r) r^{N-1}\, dr \\
\geq \lambda \int_0^1 a(r)f(u(r)) v(r) r^{N-1}\, dr,
\end{multline*}
and the proof is complete.
\end{proof}

\bigskip
\begin{proof}[Proof of Theorem \ref{main}] 
Since any maximizer $u$ for $S$ satisfies
\begin{equation*}
\int_B |\nabla u|^{p-2}\nabla u \cdot \nabla v + \int_B |u|^{p-2}u v = \lambda \int_B a(|x|) f(u)v\, dx
\end{equation*}
for any radial function $v \in W^{1,p}(B)$, we conclude that $u$ is a
positive weak solution of (\ref{1}). 
\end{proof}

\begin{remark}
The same approach solves the more general problem
\[
        \begin{cases}
        -\Delta_p u + |u|^{p-2}u = \lambda f(|x|,u) &\text{in $B$} \\
        \frac{\partial u}{\partial \nu} = 0 &\text{on $\partial B$}
\end{cases}
        \]
under reasonable assumptions on $f \colon [0,+\infty) \times [0,+\infty) \to \mathbb{R}$. For instance, one has to require that $f$ is continuous, $f(0,0)=0$, $f(r,s)$ is separately monotone in $r$ and in $s$, and that $s \mapsto f(r,s)/s^{p-1}$ is strictly increasing for every fixed $r$.
\end{remark}

\section{Solutions of (\ref{1}) when $\lambda=1$}

As we wrote at the beginning, it is in general impossible to use the weak solution $u$ of the previous section to construct a solution of the problem
\begin{equation} \label{13}
\begin{cases}
-\Delta_p u + |u|^{p-2}u = a(|x|)f(u)  &\text{in $B$}\\
u>0 &\text{in $B$}\\
\frac{\partial u}{\partial \nu} =0 &\text{on $\partial B$}.
\end{cases}
\end{equation}
However, following \cite{St}, a radially symmetric increasing solution to (\ref{13}) can be found under slightly more restrictive conditions on $f$. Besides (A), we assume
\begin{description}
\item[(F')] $f \in C^1([0,+\infty))$, $f(t)=o(t^{p-1})$ as $t \to 0$; $f'(t)t-(p-1)f(t)>0$ for all $t>0$; there exists $\gamma>p$ such that $f(t)t \geq \gamma F(t)$ for all $t>0$.
\item[(F'')] $t \mapsto f(t)/t^{p-1}$ is strictly increasing on $(0,+\infty)$.
\end{description}
\begin{remark}
It follows from (F') that
\[
f(t)t \geq C t^\gamma \quad\text{and}\quad F(t)\geq C t^\gamma \quad\text{for all $t>0$}.
\]
\end{remark}
\begin{remark}
Once more, we could treat the more general equation $-\Delta_pu + |u|^{p-2}u = f(|x|,u)$ under suitable assumptions on the right--hand side.
\end{remark}
Our main theorem about problem (\ref{13}) is the following.
\begin{theorem}\label{th:main2}
Retain assumptions (A), (F') and (F''). Then there exists a positive, radially symmetric and increasing solution to problem (\ref{13}).
\end{theorem}
We define the functional $J \colon W^{1,p}(B) \to \mathbb{R}$ by the formula
\begin{equation}
J (u) = \frac{1}{p}\|u\|^p - \int_B a(|x|) F(u).
\end{equation}
We would like to find a critical point of $J$ lying in $\mathcal{M}$, but this functional is unbounded. Therefore we introduce the set
\begin{equation}
\mathcal{N} = \left\{ u \in \mathcal{M}\setminus\{0\} \mid \|u\|^p = \int_B a(|x|)f(u)u \right\}. 
\end{equation}
Clearly, this set is an imitation of the well-known \emph{Nehari manifold} in Critical Point Theory. More precisely, it is the intersection of the standard Nehari manifold with $\mathcal{M}$, and therefore it is not homeomorphic to the unit sphere of $W^{1,p}(B)$.
\begin{lemma} \label{lem:18}
The constraint $\mathcal{N}$ is radially homeomorphic to the ball $\{u \in \mathcal{M} \mid \|u\|=1\}$.
\end{lemma}
\begin{proof}
We claim that for every $u \in \mathcal{M}\setminus\{0\}$ there exists a unique~$t>0$ such that $tu \in \mathcal{N}$. 
Indeed, we need to find $t>0$ that solves
\begin{equation}
t^p \|u\|^p = \int_B a(|x|)f(tu)tu\, dx.
\end{equation}
Define the auxiliary function $\sigma \colon [0,+\infty) \to \mathbb{R}$ by the formula 
\begin{equation}\label{eq:sigma}
\sigma(t)=t^p\|u\|^p-\int_B a(|x|)f(tu)tu\, dx.
\end{equation}
Since $f(t)=o(t^{p-1})$ as $t \to 0$ and $u \in L^\infty(B)$, for every $\varepsilon>0$ we can choose $\delta>0$ with the property that $|tu(x)| \leq \delta$ implies $0 \leq f(tu(x)) < \varepsilon |tu(x)|^{p-1}$. If $0 \leq t \leq \delta \|u\|_{L^\infty}^{-1}$, then
\begin{equation*}
0 \leq \int_B a(|x|)f(tu)tu\, dx \leq \varepsilon t^{p} \int_B a(|x|)u^{p}\, dx \leq \varepsilon t^{p} \|u\|_{L^\infty}^{p} \|a\|_{L^1} = C \varepsilon t^{p}.
\end{equation*}
Therefore $\sigma(t)\geq t^p\|u\|^p-C\varepsilon t^p>0$, provided $t$ and $\varepsilon$ are small enough. Recalling the properties of $f$,
\[
\int_B a(|x|)f(tu)tu\, dx \geq C t^\gamma \int_B a(|x|)u^\gamma\, dx,
\]
and then
\[
\limsup_{t \to +\infty} \sigma(t) \leq \limsup_{t \to +\infty} \left( t^p\|u\|^p - C t^\gamma \int_B a(|x|)u^\gamma \, dx \right)= -\infty.
\]
By the Intermediate Value Theorem, the continuous function $\sigma$ must vanish at some $t_0>0$, and this means that $t_0 u \in \mathcal{N}$. Since the map $t \mapsto f(t)/t^{p-1}$ is strictly increasing, this $t_0=t_0(u)$ is unique. 

Next, we claim that the mapping $u \mapsto t_0(u)$ is continuous. We proceed as in \cite{rab}. Assume $\{u_n\}_{n=1}^\infty$ is a sequence in $\mathcal{N}$ such that $u_n \to u$ strongly. In particular, $u \neq 0$. Then
\begin{equation} \label{eq:18}
t_0(u_n)^p \|u_n\|^p = \int_B a(|x|) f(t_0(u_n)u_n) t_0(u_n)u_n\, dx.
\end{equation}
Either $t_0(u_n) \leq 1$ or $t_0(u_n)>1$. In the latter case, we deduce that
\begin{eqnarray*}
\int_B a(|x|) f(t_0(u_n)u_n) t_0(u_n)u_n\, dx &\geq& \gamma \int_B a(|x|) F(t_0(u_n)u_n)\, dx \\
&\geq& \gamma t_0(u_n)^\gamma \int_B a(|x|) F(u_n)\, dx.
\end{eqnarray*}
Consequently,
\[
t_0(u_n)^{\gamma -p} \leq \frac{1}{\gamma} \frac{\|u_n\|^p}{\int_B a(|x|) F(u_n)\, dx} = \frac{1}{\gamma} \frac{\|u\|^p}{\int_B a(|x|) F(u)\, dx} + o(1).
\]
It follows that $\{t_0(u_n)\}_{n=1}^\infty$ is bounded from above and converges along a subsequence to a limit $t_\infty$. If $t_\infty=0$, from (\ref{eq:18}) and from the properties of $f$ we deduce that $u=0$. But this is impossibile since $u \in \mathcal{N}$. Thus $t_\infty>0$ and again from (\ref{eq:18}) we get
\[
t_\infty^{p} \|u\|^p = \int_B a(|x|) f(t_\infty u)t_\infty u\, dx. 
\]
By uniqueness, $t_\infty = t_0(u)$. By a standard argument, $t_0(u_n) \to t_0(u)$ along the whole sequence, and our claim is proved.
\end{proof}
We introduce the quantity
\begin{equation}
c_0 = \inf_{u \in \mathcal{N}} J(u).
\end{equation}
In the sequel, it may be useful to notice that
\begin{equation}
J_{|\mathcal{N}}(u) = \frac{1}{p} \int_B a(|x|) f(u)u\, dx - \int_B a(|x|)F(u)\, dx.
\end{equation}

\begin{proposition}
The level $c_0$ is attained, \ie\ there exists $u \in \mathcal{N}$ such that $J(u)=c_0$.
\end{proposition}
\begin{proof}
We claim that $c_0>0$. Indeed, from the properties of $f$ and $F$,
\begin{eqnarray} \label{19}
J(u) &=& \frac{1}{p} \|u\|^p - \int_B a(|x|)F(u)\, dx \nonumber\\
&=& \frac{1}{p} \|u\|^p - \frac{1}{\gamma} \|u\|^p + \frac{1}{\gamma} \|u\|^p - \int_B a(|x|)F(u)\, dx \nonumber\\
&=& \left( \frac{1}{p} - \frac{1}{\gamma}\right) \|u\|^p + \frac{1}{\gamma} \int_B a(|x|) \left( f(u)u-\gamma F(u)\right)\, dx \nonumber\\
&\geq& \left( \frac{1}{p} - \frac{1}{\gamma}\right) \|u\|^p.
\end{eqnarray}
We need to show that $\inf_{u \in \mathcal{N}} \|u\|^p>0$. Assume the existence of a sequence $\{u_n\}_{n=1}^\infty \subset \mathcal{N}$ such that $u_n \to 0$ strongly. Then
\begin{eqnarray*}
\|u_n\|^p &=& \int_B a(|x|)f(u_n)u_n\, dx \leq f(\|u_n\|_{L^\infty}) \|u_n\|_{L^\infty} \|a\|_{L^1} \\
&\leq& f(C_{\mathcal{M}}\|u_n\|) C_{\mathcal{M}} \|u_n\| \|a\|_{L^1},
\end{eqnarray*}
where $C_{\mathcal{M}}>0$ is the constant of Proposition \ref{prop:infinity}. But $f(t)=o(t^{p-1})$ as $t \to 0$, and this leads to a contradiction. Our claim is proved.

Now, let $\{u_n\}_{n=1}^\infty \subset \mathcal{N}$ be a minimizing sequence for $c_0$. From (\ref{19}) it follows that $\{u_n\}_{n=1}^\infty$ is bounded in $W^{1,p}(B)$, and we can assume without loss of generality that $u_n \to u$ pointwise almost everywhere and weakly in $W^{1,p}(B)$. It is easy to check that $u \in \mathcal{M}$, and $\sup_n \|u_n\|_{L^\infty} < \infty$. By Lebesgue's theorem on Dominated Convergence,
\begin{eqnarray}
c_0+o(1) &=& J(u_n) = \frac{1}{p}\int_B a(|x|)f(u_n)u_n \, dx - \int_B a(|x|)F(u_n)\, dx \nonumber \\
&=& \frac{1}{p}\int_B a(|x|)f(u)u \, dx - \int_B a(|x|)F(u)\, dx + o(1).
\end{eqnarray}
We deduce that $u \neq 0$, since $c_0>0$. We have two cases:
\begin{enumerate}
\item $u \in \mathcal{N}$. This implies that $\int_B a(|x|)f(u)u\, dx = \|u\|^p$, and the proof is complete.
\item $u \notin \mathcal{N}$. By the lower semicontinuity of the norm, we have 
\begin{equation}\label{21}
\|u\|^p < \int_B a(|x|)f(u)u\, dx.
\end{equation}
Consider again the auxiliary function $\sigma$ defined in~(\ref{eq:sigma}). Since $u \neq 0$, $\sigma(t)>0$ for $t$ positive and small. By (\ref{21}), $\sigma(1)<0$, and hence $\sigma$ vanishes at some $t_0 \in (0,1)$. We know that $t_0 u \in \mathcal{N}$, and the properties of $f$ and $F$ imply that $t \mapsto (1/p)f(t)t-F(t)$ is strictly increasing: just differentiate and use (F'). Since $t_0<1$, we deduce
\begin{eqnarray*}
c_0 &\leq& J(t_0u) = \int_B a(|x|) \left(  \frac{1}{p} f(t_0 u)t_0 u - F(t_0 u)\right)\, dx \\
&<& \int_B a(|x|) \left(  \frac{1}{p} f(u)u - F(u)\right)\, dx = c_0.
\end{eqnarray*}
This contradiction proves that (\ref{21}) cannot hold.
\end{enumerate}
Since we have shown that $u \in \mathcal{N}$, the proof is complete.
\end{proof}

We are now ready to apply all the arguments of Section \ref{sec:3} to our minimizer $u \in \mathcal{N}$. The only difference is that we need an analog of Proposition \ref{prop:10}. 

\begin{proposition}
Let $u$ be a minimizer of $J$ on $\mathcal{N}$. Assume that $v \in W^{1,p}(B)$ is a radial function such that $u+s v \in \mathcal{M}$ for every positive small $s$. Then
\begin{equation}\label{22}
\int_B |\nabla u|^{p-2} \nabla u \cdot \nabla v + \int_B |u|^{p-2}uv \geq \int_B a(|x|) f(u)uv.
\end{equation}
\end{proposition}
\begin{proof}
Without loss of generality, we assume that $u+s v$ does not vanish identically. By Lemma \ref{lem:18}, to every small $s\geq 0$ we can attach some $t=t(s)$ with $t(s)(u+sv) \in \mathcal{N}$. Define $G(s,t)=t^p \|u+sv\|^p - \int_B a(|x|) f(t(u+sv))t(u+sv)\, dx$. By definition, $G(0,1)=0$ since $u \in \mathcal{N}$. In addition,
\begin{eqnarray*}
\frac{\partial G}{\partial t}(0,1) &=& p \|u\|^p - \int_B a(|x|) f'(u)u^2 \, dx - \int_B a(|x|) f(u)u\, dx \\
&=& (p-1)\int_B a(|x|) f(u)u\, dx - \int_B a(|x|) f'(u)u^2 \, dx \\
&=& \int_B a(|x|) \left[ (p-1)f(u) - f'(u)u\right]u \, dx <0.
\end{eqnarray*}
The Implicit Function Theorem yields the existence of some $\delta>0$ and of a $C^1$ function $t \colon [0,\delta) \to \mathbb{R}$ such that $t(0)=1$ and
\[
G(s,t(s))=0\quad\text{for all $s \in [0,\delta)$}.
\]
By definition, this means that $t(s)(u+sv) \in \mathcal{N}$ for all $s \in [0,\delta)$. The function $H(s) = J(t(s)(u+sv))$, defined on $[0,\delta)$, has a local minimum point at $s=0$. Therefore
\begin{eqnarray*}
0 &\leq& H'(0) \\
&=& t'(0) \|u\|^{p} + \int_B |\nabla u|^{p-2}\nabla u \cdot \nabla v \\
&&\qquad{}+ \int_B |u|^{p-2}uv - t'(0)\int_B a(|x|)f(u)u\, dx - \int_B a(|x|) f(u)v \, dx \\
&=& t'(0) \left( \|u\|^{p} - \int_B a(|x|) f(u)u \, dx
\right) \\
&&\qquad{}+ \int_B |\nabla u|^{p-2}\nabla u \cdot \nabla v + \int_B |u|^{p-2}uv- \int_B a(|x|) f(u)v \, dx.
\end{eqnarray*}
Since $u \in \mathcal{N}$, the big bracket vanishes and (\ref{22}) is proved.
\end{proof}
\begin{proof}[Proof of Theorem \ref{th:main2}]
Since we have a function $ \in \mathcal{N}$ that satisfies (\ref{22}),
we can apply to $u$ the same arguments developed in Section
\ref{sec:3} and conclude that $u$ is a weak solution of (\ref{13}).
\end{proof}
 
\section{Final remarks}

As we have seen, the variational technique is rather elementary, and can be adapted to more general equations. In both sections, the crucial point was that the tentative solution $u$ was not \emph{a priori} surrounded by a neighborhood contained in $\mathcal{M}$, and it was not clear how to compute the Gateaux derivative at $u$.

Anyway, since our solution is a true (constrained) critical point of the associated functional, we may ask ourselves what kind of critical point it is. Its variational characterization is rather subtle, since is maximizes (or minimizes) a functional only among the set of radially increasing functions. It could be interesting to look for a more precise description of this solution.

\bigskip

\noindent
\begin{grazie}
The author is grateful to E.~Serra for suggesting this problem.
\end{grazie}


    \bibliographystyle{amsplain}
    \bibliography{Bibliography}

   \end{document}